\documentclass[12pt]{amsart}

\usepackage{amssymb,bm} % \mathbb
\usepackage{tikz}
\usepackage[normalem]{ulem}
\usepackage{CJKutf8} %日本語でコメントを入れるパッケージ
\numberwithin{equation}{section}

\newtheorem{theorem}{Theorem}[section]

\newtheorem{lemma}[theorem]{Lemma}
\newtheorem{corollary}[theorem]{Corollary}

\theoremstyle{definition}

\newtheorem{example}[theorem]{Example}

\newtheorem{conjecture}[theorem]{Conjecture}
\newtheorem{remark}[theorem]{Remark}

\newcommand{\ZZ}{ \ensuremath{\mathbb{Z}}}

\newcommand{\init}{\ensuremath{\mathop{\mathrm{in}}}}

\newcommand{\mideal}{\ensuremath{\mathfrak{m}}}

\newcommand{\PF}{\mathrm{PF}}
\newcommand{\socle}{\mathrm{socle}}
\newcommand{\F}{\mathbb{K}}
\newcommand{\Ap}{\mathrm{Ap}}

\begin{document}

\title[A criterion for determinantal presentations]{A criterion for determinantal presentations of numerical semigroup rings}

\author{Satoshi Murai}
\address{
Satoshi Murai,
Department of Mathematics
Faculty of Education
Waseda University,
1-6-1 Nishi-Waseda, Shinjuku, Tokyo 169-8050, Japan}
\email{s-murai@waseda.jp}

\author{Kou Takahashi}
\address{
Kou Takahashi,
Department of Mathematics
Faculty of Education
Waseda University,
1-6-1 Nishi-Waseda, Shinjuku, Tokyo 169-8050, Japan}
\email{k\_takahash@akane.waseda.jp}

%\author{Another Author}
%\address{
%}
%\email{}

%Keyword and Subject Classes (if needed)
%\keywords{}
%\subjclass[2000]{}
%\dedicatory{Dedicated to on the occasion of his birthday}

\begin{abstract}
We study numerical semigroups whose defining ideals admit determinantal presentations by the $2\times2$ minors of a $2\times n$ matrix. A conjecture of Cuong, Kien, Matsuoka, and Truong relates such presentations to arithmetic progressions of pseudo-Frobenius numbers. We prove a numerical criterion for this determinantal presentation, reducing the conjecture to a problem on factorizations of
maximal elements of an Ap\'ery set in the numerical semigroup. As an application, we verify the conjecture in embedding dimension four under a unique factorization assumption on an Ap\'ery set.
\end{abstract}

\maketitle

\section{Introduction}

In this paper, we study numerical semigroups whose defining ideals are generated by the $2\times 2$ minors of a $2\times n$ matrix.
We begin by recalling some basic notation and background; we refer to \cite{RGS} for general background on numerical semigroups. For positive integers $a_1,\dots,a_n$ with $\gcd(a_1,\dots,a_n)=1$, let
\[
\langle a_1,\dots,a_n\rangle
=
\left\{
\lambda_1a_1+\cdots+\lambda_n a_n
\;\middle|\;
\lambda_1,\dots,\lambda_n\in \mathbb Z_{\geq 0}
\right\}
\]
be the numerical semigroup generated by $a_1,\dots,a_n$.
%We always assume that $\gcd(a_1,\dots,a_n)=1$.
Fix a field $\F$. If
$H=\langle a_1,\dots,a_n\rangle$
is minimally generated by $a_1,\dots,a_n$, then the defining ideal $I_H$ of the numerical semigroup ring $\F[H]$ is the kernel of the homomorphism
\[
\pi:\F[x_1,\dots,x_n]\longrightarrow \F[H]=\F[t^{a_1},\dots,t^{a_n}],
\qquad
\pi(x_i)=t^{a_i}.
\]
Determining generators of the defining ideal of a numerical semigroup ring is a fundamental problem in the study of numerical semigroups; see, for example, \cite{MS} for an overview of problems and recent developments.
A classical result of Herzog \cite{Herzog} describes the defining ideals of numerical semigroups of embedding dimension three. In particular, if $H$ is generated by three elements and is not symmetric, then $I_H$ is generated by the $2\times2$ minors of a $2\times3$ matrix.
Recently, Cuong, Kien, Matsuoka, and Truong proposed the following conjectural characterization of higher-dimensional numerical semigroups with such a determinantal defining ideal; see \cite[Conjecture 1.1]{KMO}. Recall that an integer
$w\in \mathbb Z_{\geq0}\setminus H$
is called a \emph{pseudo-Frobenius number} of $H$ if
$w+a_i\in H$
for every $i$. We denote by $\PF(H)$ the set of pseudo-Frobenius numbers of $H$.
For a matrix $M$, we write $I_2(M)$ for the ideal generated by its
$2$-minors.

\begin{conjecture}[Cuong--Kien--Matsuoka--Truong]
\label{conj}
Let $n\geq3$, and let
$H=\langle a_1,\dots,a_n\rangle$
be a numerical semigroup minimally generated by $a_1,\dots,a_n$. Then the following conditions are equivalent.
\begin{itemize}
    \item[(A)] $I_H$ is generated by the $2\times2$ minors of a
$2\times n$  whose entries are homogeneous elements of positive degree.
    \item[(B)] There exist positive integers $c_1,\dots,c_n,d_1,\dots,d_n$ such that
    \begin{align}
    \label{2minorideal}
    I_H
    =
    I_2
    \begin{pmatrix}
    x_1^{c_1} & \cdots & x_{n-1}^{c_{n-1}} & x_n^{c_n}\\
    x_2^{d_2} & \cdots & x_n^{d_n} & x_1^{d_1}
    \end{pmatrix}.
    \end{align}
    \item[(C)] $\PF(H)$ forms an arithmetic progression of length $n-1$.
\end{itemize}
\end{conjecture}

The implication $(A)\Rightarrow(C)$ was proved in \cite{KM}; hence the open part of the conjecture is basically the implication $(C)\Rightarrow(B)$. The conjecture is known in several special cases, including almost symmetric numerical semigroups \cite{GKMT}, numerical semigroups whose embedding dimension is large relative to their multiplicity \cite{KM,Ta}, and stretched numerical semigroups \cite{KMO}.
Algebraic properties of numerical semigroup rings with determinantal defining ideals as in \eqref{2minorideal} have also studied in \cite{KMN}.

A difficulty in proving $(C)\Rightarrow(B)$ is that, even after suitable candidates for the exponents
$c_1,\dots,c_n,d_1,\dots,d_n$
have been identified, one still has to prove that the corresponding determinantal ideal is indeed the defining ideal $I_H$. Our first main result gives a numerical criterion that removes this difficulty.

To explain the result, let us recall a consequence of condition $(B)$. It was proved in \cite[Theorem 3]{KM} that, if \eqref{2minorideal} holds, then
\[
\PF(H)
=
\{h+\alpha,h+2\alpha,\dots,h+(n-1)\alpha\},
\]
where
\[
h=\sum_{i=1}^n(c_i-1)a_i
\text{ and }
\alpha=d_{i+1}a_{i+1}-c_ia_i \text{ for all 
$i$},
\]
where we consider $d_{n+1}=d_1$ and $a_{n+1}=a_1$. Moreover, one obtains explicit numerical expressions of pseudo-Frobenius numbers in terms of the exponents $c_i$ and $d_i$:
\begin{align}
\label{PFformula}
h+k\alpha
=
-(a_1+\cdots+a_n)
+
\sum_{\ell=2}^{k+1}d_\ell a_\ell
+
\sum_{\ell=k+1}^{n}c_\ell a_\ell
\end{align}
for $k=0,1,2,\dots,n$.
Thus condition $(B)$ imposes very explicit numerical identities on the pseudo-Frobenius numbers. A natural question is whether, conversely, these numerical identities are sufficient to recover the determinantal presentation of $I_H$. Our first main theorem shows that this is indeed the case.

\begin{theorem}
\label{maincor}
Let
$H=\langle a_1,\dots,a_n\rangle$
be a numerical semigroup minimally generated by $a_1,\dots,a_n$, and assume that
\[
\PF(H)
=
\{h+\alpha,h+2\alpha,\dots,h+(n-1)\alpha\}
\]
for some $h,\alpha\in\mathbb Z$, where $\alpha \ne 0$. Let
$c_1,\dots,c_n,d_1,\dots,d_n$
be positive integers. Then the following conditions are equivalent.
\begin{itemize}
    \item[(1)]
    $I_H
    =
    I_2
    \begin{pmatrix}
    x_1^{c_1} & \cdots & x_{n-1}^{c_{n-1}} & x_n^{c_n}\\
    x_2^{d_2} & \cdots & x_n^{d_n} & x_1^{d_1}
    \end{pmatrix}.
    $
    \item[(2)] The identities \eqref{PFformula} hold for all $k=0,1,2,\dots,n$.
\end{itemize}
\end{theorem}

The significance of Theorem \ref{maincor} is that it separates the numerical-semigroup-theoretic part of Conjecture \ref{conj} from the commutative-algebraic verification of the defining ideal. 
In the proofs of the previously known cases of the conjecture \cite{GKMT,KM,KMO,Ta}, one first determines a candidate determinantal presentation and then proves separately that the resulting determinantal ideal coincides with $I_H$.
Theorem \ref{maincor} shows that the latter step follows automatically once the appropriate numerical identities for the pseudo-Frobenius numbers have been established.

In the second part of the paper,
as an application of Theorem \ref{maincor}, we obtain a new partial affirmative result for Conjecture \ref{conj} in embedding dimension four.
For $w\in H$, we say that $w$ has a \emph{unique factorization} in $H$ if there is a unique tuple
$(k_1,\dots,k_n)\in\mathbb Z_{\geq0}^n$
such that
$
w=k_1a_1+\cdots+k_n a_n.
$
For a nonzero element $a\in H$, let
\[
\Ap(H,a)
=
\{w\in H\mid w-a\notin H\}
\]
be the Ap\'ery set of $H$ with respect to $a$.
We say that $\Ap(H,a)$ has unique factorizations in $H$ if every element of $\Ap(H,a)$ has a unique factorization in $H$.
%This type of uniqueness condition on Ap'ery sets has been studied previously in \cite{Rosales}.
Our second main result is the following.

\begin{theorem}
\label{secondthm}
Conjecture \ref{conj} holds when $n=4$ and $\Ap(H,a_i)$ has unique factorizations in $H$ for some $i$.
\end{theorem}

At first sight, the assumption that $\Ap(H,a_i)$ has unique factorizations may appear restrictive. However, among semigroups satisfying condition $(B)$, this condition is fairly mild in embedding dimension four. Indeed, as shown in Lemma \ref{uniqueAP}, if
\[
I_H
=
I_2
\begin{pmatrix}
x_1^{c_1} & x_2^{c_2} & x_3^{c_3} & x_4^{c_4}\\
x_2^{d_2} & x_3^{d_3} & x_4^{d_4} & x_1^{d_1}
\end{pmatrix},
\]
then $\Ap(H,a_1)$ has unique factorizations if and only if
\[
d_2\leq c_2
\qquad\text{and}\qquad
c_4\leq d_4.
\]
Thus the failure of the unique-factorization condition is governed by simple inequalities among the exponents. In this sense, assuming Conjecture \ref{conj}, Theorem \ref{secondthm} applies to a substantial part of the embedding-dimension-four case.
We also note that, in the previously known cases of the conjecture \cite{GKMT,KM,KMO,Ta}, apart from Herzog's classical three-generated case \cite{Herzog}, at least half of the exponents $c_1,\dots,c_n,d_1,\dots,d_n$ are equal to $1$. Our result does not require such a restriction.

This paper is organized as follows.
In Section 2, we recall some basic facts on numerical semigroups,
Ap\'ery sets, and Gr\"obner bases.
In Section 3, we establish Theorem \ref{maincor}, our numerical criterion
for determinantal presentations.
In Section 4, we study the embedding-dimension-four case and prove
Theorem \ref{secondthm}.
We conclude with some remarks and examples.

\section{Preliminary}

In this section, we recall some basic facts about Ap\'ery sets,
pseudo-Frobenius numbers, numerical semigroup rings, and Gr\"obner bases.

Throughout the paper for integer vectors $v=(v_1,\dots,v_t),u=(u_1,\dots,u_t)$
we write $v\geq u$ if $v_i \geq u_i$ for all $i$.
Let $\leq_H$ be the partial order on $\ZZ$ defined by
$a\leq_H b$ if $b-a\in H$.
We recall the following standard properties of Ap\'ery sets;
see \cite[Lemma 2.4 and Proposition 2.20]{RGS}.

\begin{lemma}
\label{aperibasic}
Let $H=\langle a_1,\dots,a_n\rangle$ be the numerical semigroup
minimally generated by $a_1,\dots,a_n$, and let $0\neq a\in H$.
Then the following statements hold.
\begin{enumerate}
    \item
    $a+\PF(H)=\{a+p\mid p\in\PF(H)\}$ is exactly the set of maximal
    elements of $\Ap(H,a)$ with respect to the partial order $\leq_H$.
    \item
    For each $i=0,1,\dots,a-1$, there is a unique element
    $w\in\Ap(H,a)$ such that $w\equiv i\pmod a$.
    \item
    If $w\in\Ap(H,a)$ and $u\leq_H w$, then $u\in\Ap(H,a)$.
\end{enumerate}
\end{lemma}

We next recall the algebraic interpretations of Ap\'ery sets and
pseudo-Frobenius numbers.
Let $H=\langle a_1,\dots,a_n\rangle$ be a numerical semigroup
minimally generated by $a_1,\dots,a_n$, and let
$S=\F[x_1,\dots,x_n]$ be the polynomial ring with grading
$\deg x_i=a_i$ for all $i$.
The numerical semigroup ring of $H$ over $\F$ is
\[
\F[H]=\F[t^{a_1},\dots,t^{a_n}]\subseteq\F[t].
\]
Its defining ideal $I_H\subseteq S$ is the kernel of the graded
homomorphism
\[
\pi:S\longrightarrow\F[H]
\]
defined by $\pi(x_i)=t^{a_i}$ for all $i$.
The following statement follows directly from the definition of an
Ap\'ery set.

\begin{lemma}
\label{aperibasic2}
With the notation above,
$x_1^{k_1}\cdots x_n^{k_n}$ is nonzero in $S/(I_H+(x_i))$
if and only f
$w=k_1a_1+\cdots+k_na_n\in H$ belongs to $\Ap(H,a_i)$.
\end{lemma}

We also need the following fact.

\begin{lemma}[see {\cite[Proposition 2.10]{KMO}}]
\label{lem6}
Let $H=\langle a_1,\dots,a_n\rangle$ be a numerical semigroup
generated by $a_1,\dots,a_n$.
Let $J$ be a graded ideal of $S$ such that $J\subseteq I_H$.
If $J+(x_1)=I_H+(x_1)$, then $J=I_H$.
\end{lemma}

We now recall the relation between pseudo-Frobenius numbers and socles.
For a graded $\F$-algebra $A$ with its graded maximal ideal $\mideal_A$,
its {\bf socle} is 
\[
\socle(A)=\{f\in A\mid \mideal_A f=0\}.
\]
We will also use the following elementary observation.

\begin{lemma}
\label{lem5}
Let $I$ be a graded ideal of $S$ such that
$\dim_\F(S/I)<\infty$.
For every nonzero element $f$ of $S/I$, there is an element
$m\in S$ such that
$0\neq mf\in\socle(S/I)$.
\end{lemma}

Socles are closely related to pseudo-Frobenius numbers.
The following result may be viewed as an algebraic analogue of
Lemma \ref{aperibasic}(1).

\begin{lemma}[see {\cite[\S 5.1.1]{HMR}}]
\label{lem7}
With the notation above, the set
\[
\{t^w\mid w\in a_i+\PF(H)\}
\]
forms a $\F$-basis of
$\socle(\F[H]/(t^{a_i}))$.
\end{lemma}

We next recall some basic facts about Gr\"obner bases.
We refer the reader to \cite[\S 15]{Ei} and \cite[Chapter 2]{HH} for the basic theory of
Gr\"obner bases.
Let $>_{\mathrm{rev}}$ be the weighted graded reverse lexicographic
order on $S$ with $x_1<\cdots<x_n$.
Thus, for monomials
$u=x_1^{k_1}\cdots x_n^{k_n}$ and
$v=x_1^{\ell_1}\cdots x_n^{\ell_n}$, one has
$u>_{\mathrm{rev}}v$ if and only if
\begin{enumerate}
    \item[(i)] $\deg u>\deg v$, or
    \item[(ii)] $\deg u=\deg v$ and the leftmost nonzero entry of
    $(\ell_1-k_1,\dots,\ell_n-k_n)$ is positive.
\end{enumerate}
Here the grading on $S$ is given by $\deg x_i=a_i$.
For a polynomial $0 \ne f\in S$ and an ideal $I\subset S$, we write
$\init(f)$ and $\init(I)$ for the initial monomial of $f$ and the
initial ideal of $I$ with respect to $>_{\mathrm{rev}}$, respectively.
A subset $\mathcal G\subset I$ is called a {\bf Gr\"obner basis} of
$I$ if
\[
\init(I)=(\init(f)\mid f\in\mathcal G).
\]

\begin{lemma}[{see \cite[Proposition 15.12]{Ei}}]
\label{lem1}
For any graded ideal $I$ of $S$, one has
\[
\init(I+(x_1))=\init(I)+(x_1).
\]
\end{lemma}

For a monomial ideal $I$ of $S$, we write $\mathcal M(I)$ for the set
of monomials in $S$ that are not contained in $I$.

\begin{lemma}[{see \cite[Proposition 2.2.5]{HH}}]
\label{lem2}
For any ideal $I$ of $S$, the set $\mathcal M(\init(I))$ forms an
$\F$-basis of $S/I$.
\end{lemma}

We will also use the following inequality comparing the socles of
$S/I$ and $S/\init(I)$.

\begin{lemma}[{see \cite[Theorem 3.3.1]{HH}}]
\label{lem3}
For any graded ideal $I$ of $S$, one has
\[
\dim_\F\big(\socle(S/I)\big)
\leq
\dim_\F\big(\socle(S/\init(I))\big).
\]
\end{lemma}

We finally discuss when $\Ap(H,a_1)$ has unique factorizations in $H$
in the case where $I_H$ has the determinantal presentation
\eqref{2minorideal}.
Assume that
$H=\langle a_1,a_2,a_3,a_4\rangle$
is the numerical semigroup minimally generated by $a_1,\dots,a_4$ and
\[
I_H=
I_2
\begin{pmatrix}
x_1^{c_1} & x_2^{c_2} & x_3^{c_3} & x_4^{c_4}\\
x_2^{d_2} & x_3^{d_3} & x_4^{d_4} & x_1^{d_1}
\end{pmatrix}.
\]
As follows from \eqref{PFformula}, we have
\[
\PF(H)=\{p_1=h+\alpha,p_2=h+2\alpha,p_3=h+3\alpha\},
\]
where
\begin{align}
\label{4presentation}
\begin{array}{ll}
p_1
&=h+\alpha
=-a_1+(c_2+d_2-1)a_2+(c_3-1)a_3+(c_4-1)a_4,\\
p_2
&=h+2\alpha
=-a_1+(d_2-1)a_2+(c_3+d_3-1)a_3+(c_4-1)a_4,\\
p_3
&=h+3\alpha
=-a_1+(d_2-1)a_2+(d_3-1)a_3+(c_4+d_4-1)a_4.
\end{array}
\end{align}

\begin{lemma}
\label{uniqueAP}
With the same notation as above, one has
\begin{enumerate}
    \item
    $a_1+p_1$ has a unique factorization in $H$ if and only if
    $c_4\leq d_4$.
    \item
    $a_1+p_2$ has a unique factorization in $H$ if and only if
    $c_4\leq d_4$ or $d_2\leq c_2$.
    \item
    $a_1+p_3$ has a unique factorization in $H$ if and only if
    $d_2\leq c_2$.
\end{enumerate}
In particular, $\Ap(H,a_1)$ has unique factorizations in $H$ if and
only if $c_4\leq d_4$ and $d_2\leq c_2$.
\end{lemma}

\begin{proof}
Recall that an element $w\in H$ belongs to $\Ap(H,a_1)$ if and only if
it has no factorization of the form
$
w=k_1a_1+k_2a_2+k_3a_3+k_4a_4
$
with $k_1,\dots,k_4\in\mathbb Z_{\geq0}$ and $k_1>0$.
Modulo $(x_1)$, the only generator of $I_H+(x_1)$ that can identify
two nonzero monomials is
\[
x_3^{c_3+d_3}-x_2^{c_2}x_4^{d_4}.
\]
Indeed, all the other $2\times2$ minors reduce modulo $(x_1)$ to
monomials, which vanish in $S/(I_H+(x_1))$.
Hence, if
\[
w=k_2a_2+k_3a_3+k_4a_4\in\Ap(H,a_1),
\]
then the monomial $x_2^{k_2}x_3^{k_3}x_4^{k_4}$ admits another
monomial representative of the same $H$-degree if and only if it is
divisible by either $x_3^{c_3+d_3}$ or $x_2^{c_2}x_4^{d_4}$.
Equivalently, $w$ has more than one factorization in $H$ if and only if
$k_3\geq c_3+d_3$ or $(k_2,k_4) \geq (c_2,d_4)$.
%\[
%b_3\geq c_3+d_3
%\quad\text{or}\quad
%b_2\geq c_2\ \text{and}\ b_4\geq d_4.
%\]
Applying this criterion to \eqref{4presentation}, we obtain
statements (1)--(3).

Finally, $\Ap(H,a_1)$ has unique factorizations in $H$ if and only if
its maximal elements $a_1+p_1,a_1+p_2,a_1+p_3$ have unique
factorizations in $H$.
Indeed, if $w\leq_H v$ and $w$ has two distinct factorizations, then
any factorization of $v-w$ extends them to two distinct factorizations
of $v$.
\end{proof}

\section{Proof of Theorem \ref{maincor}}

In this section, after establishing a few preliminary lemmas, we prove
our first main theorem.
Throughout this section, let
$H=\langle a_1,\dots,a_n\rangle$ be a numerical semigroup minimally
generated by $a_1,\dots,a_n$, and let
$S=\F[x_1,\dots,x_n]$ with $\deg x_i=a_i$.
We also fix positive integers
$c_1,\dots,c_n,d_1,\dots,d_n$ and write
    \[J=I_2 \begin{pmatrix} x_1^{c_1} & \cdots & x_{n-1}^{c_{n-1}} & x_n ^{c_n} \\ x_2^{d_2} & \cdots & x_{n}^{d_{n}} & x_1 ^{d_1}\end{pmatrix}.\]
To simplify notation, we set
$x_{n+1}=x_1$, $c_{n+1}=c_1$, and $d_{n+1}=d_1$.
We note that $J$ is homogeneous with respect to the grading
$\deg x_i=a_i$ if and only if
$
c_i a_i-d_{i+1}a_{i+1}
$ is independent of $i$.

\begin{lemma}
    \label{lem2.0}
If $J$ is a graded ideal of $S$, then the set
\[\mathcal G=\{x_{i+1}^{d_{i+1}}x_j^{c_j} -x_i^{c_i}x_{j+1}^{d_{j+1}} \mid 1 \leq i < j \leq n\}\]
is a Gr\"obner basis of $J$ with respect to $>_{\mathrm{rev}}$.
\end{lemma}

\begin{proof}
We use Buchberger's criterion (see \cite[Theorem 2.3.2]{HH}).
We write
$f \xrightarrow{g} h$
if $h$ is obtained from $f$ by a single polynomial reduction using $g$,
and
$f \xrightarrow{\mathcal G} h$
if $f$ reduces to $h$ by a finite sequence of polynomial reductions
using elements of $\mathcal G$.
For $1 \leq i<j \leq n$, let
\[f_{i,j}=x_{i+1}^{d_{i+1}}x_j^{c_j}-x_i^{c_i} x_{j+1}^{d_{j+1}}.\]
Thus $\mathcal G=\{f_{i,j} \mid 1 \leq i< j \leq n\}$.
It suffices to prove that
\[
S(f_{i,j},f_{k,\ell}) \stackrel {\mathcal G} \longrightarrow 0
\]
for all $1\leq i<j\leq n$ and $1 \leq k<\ell \leq n$.
To prove this we may assume $i \leq k$.

Since each $f_{i,j}$ is homogeneous by our assumption,
we have
\[\init(f_{i,j})=x_{i+1}^{d_{i+1}}x_j^{c_j}.\]
If $\init(f_{i,j})$ and $\init(f_{k,\ell})$ are relatively prime then $S(f_{i,j},f_{k,\ell})$ reduces to $0$ (see \cite[Lemma 2.3.1]{HH}).
Under the assumption $i\leq k$, if
$\init(f_{i,j})$ and $\init(f_{k,\ell})$ are not relatively prime,
then either $j=\ell$, $i=k$, or $k=j-1$.
Hence it suffices to consider the following cases.

\textbf{Case (a).} Suppose that $j=\ell$ or $i=k$.
If $j=\ell$,
then we have 
\begin{align*}
S(f_{i,j},f_{k,j})
=x_{k+1}^{d_{k+1}}f_{i,j}-x_{i+1}^{d_{i+1}}f_{k,j}
%=-x_i^{c_i}x_{k+1}^{d_{k+1}}x_{j+1}^{d_{j+1}}
%+x_k^{c_k}x_{i+1}^{d_{i+1}}x_{j+1}^{d_{j+1}}
=x_{j+1}^{d_{j+1}}
\left(x_{i+1}^{d_{i+1}}x_k^{c_k}
-x_i^{c_i}x_{k+1}^{d_{k+1}}\right)
\stackrel {f_{i,k}}\longrightarrow 0
\end{align*}
as desired.
The case $i=k$ is similar.
In this case we may assume $j<\ell$ by exchanging the ordering of $f_{i,j}$ and $f_{i,\ell}$ if necessary,
and we have
\[
S(f_{i,j},f_{i,\ell})
=-x_i^{c_i} (x_{j+1}^{d_{j+1}}x_\ell^{c_\ell}-x_j^{c_j}x_{\ell+1}^{d_{\ell+1}})
\stackrel{f_{j,\ell}} \longrightarrow 0.
\]

\textbf{Case (b).} Suppose $k=j-1$.
By Case (a) we may assume $i<j-1$ and $j<\ell$.
Assume $c_j \geq d_j$.
In this case we have
\[
S(f_{i,j},f_{k,\ell})=x_\ell^{c_\ell}f_{i,j}
-x_{i+1}^{d_{i+1}}x_j^{c_j-d_j}f_{j-1,\ell}
=
x_{i+1}^{d_{i+1}}
x_{j-1}^{c_{j-1}}x_j^{c_j-d_j}x_{\ell+1}^{d_{\ell+1}}
-x_i^{c_i}x_{j+1}^{d_{j+1}}x_\ell^{c_\ell}.
\]
Observe that $f_{j,\ell}=x_{j+1}^{d_{j+1}}x_\ell^{c_\ell}
-x_j^{c_j}x_{\ell+1}^{d_{\ell+1}}$
and
$f_{i,j-1}
=x_{i+1}^{d_{i+1}}x_{j-1}^{c_{j-1}}
-x_i^{c_i}x_j^{d_j}$.
One can see
\begin{align*}
S(f_{i,j},f_{k,\ell})
\stackrel{f_{i,j-1}}\longrightarrow
x_i^{c_i}x_j^{d_j}x_j^{c_j-d_j}
x_{\ell+1}^{d_{\ell+1}}
-x_i^{c_i}x_\ell^{c_\ell}x_{j+1}^{d_{j+1}}
\stackrel{f_{j,\ell}}\longrightarrow 0
\end{align*}
as desired.
The case $c_j <d_j$ is treated similarly. Indeed, in this case, we have
\[
S(f_{i,j},f_{k,\ell})
=x_{i+1}^{d_{i+1}}x_{j-1}^{c_{j-1}}x_{\ell+1}^{d_{\ell+1}}
-x_i^{c_i}x_j^{d_j-c_j}x_{j+1}^{d_{j+1}}x_\ell^{c_\ell},
\]
and it is easy to see that this again reduces to zero, first by $f_{i,j-1}$ and then by $f_{j,\ell}$.
\end{proof} 

We next use Lemma \ref{lem2.0} to describe a $\F$-basis of
$S/(J+(x_1))$.
Let $\mathrm{Mon}(S)$ be the set of all monomials of $S$.
%For a monomial ideal $I$ of $S$, we write
%\[\mathcal M(I)=\{ u \in \mathrm{Mon}(S) \mid u \not \in I \}.\]
For monomials $m_1,\dots,m_t \in S$,
we write 
\[\langle m_1,\dots,m_t \rangle
=\{ m \in \mathrm{Mon}(S) \mid m \mbox{ divides } m_k \mbox{ for some }k\}.
\]
For $k=2,3,\dots,n$, we define
\[
u_k=\left( \prod_{2 \leq i \leq k} x_i^{d_i-1} \right) x_k \left( \prod_{k \leq j \leq n} x_j^{c_j-1} \right).
\]

\begin{lemma}
    \label{lem2.1}
    With the same notation as in Lemma \ref{lem2.0}, one has
    \[\mathcal M\big(\init(J+(x_1))\big)= \langle u_2,\dots,u_n\rangle.\]
\end{lemma}

\begin{proof}
By Lemmas \ref{lem1} and \ref{lem2.0},
\[
\init(J+(x_1))
=
(x_1)+
(x_j^{c_j}x_{i+1}^{d_{i+1}}
\mid 1\leq i<j\leq n).
\]
It is immediate from the definition that none of
$u_2,\dots,u_n$ belongs to $\init(J+(x_1))$.
Hence
\[
\langle u_2,\dots,u_n\rangle
\subseteq
\mathcal M(\init(J+(x_1))).
\]

For the reverse inclusion, let
$m=x_2^{k_2}\cdots x_n^{k_n}
\in\mathcal M(\init(J+(x_1)))$.
If $k_i<d_i$ for every $i=2,\dots,n$, then $m$ divides $u_n$.
Otherwise, let $\ell$ be the smallest integer such that $k_\ell\geq d_\ell$.
Then
\[
k_i\leq d_i-1\qquad (2\leq i<\ell).
\]
Since $x_\ell^{d_\ell}x_j^{c_j}\in\init(J)$
for $j=\ell,\dots,n$,
we also have
\[
k_\ell\leq d_\ell+c_\ell-1
\quad\text{and}\quad
k_j\leq c_j-1\qquad (\ell<j\leq n).
\]
Therefore $m$ divides $u_\ell$.
\end{proof}
%\begin{proof}
%Since $\init(J+(x_1))=\init(J)+(x_1)$ by Lemma \ref{lem1}
%and since Lemma \ref{lem2.0} says
%\[\init(J)=(x_j^{c_j}x_{i+1}^{d_{i+1}} \mid 1 \leq i < j \leq n ),\]
%it is clear that $u_2,\dots,u_n$ are not contained in $\init(J+(x_1))$.
%This proves $\mathcal M(J+(x_1)) \supset \langle u_2,\dots,u_n \rangle$.
%
%We prove
%$\mathcal M\big(\init(J+(x_1))\big) \subset \langle u_2,\dots,u_n \rangle$.
%Let $m=x_2^{b_2} \cdots x_n^{b_n} \in \mathcal M\big(\init(J+(x_1))\big).$
%We will prove $m \in \langle u_2,\dots,u_n \rangle.$
%If $b_i < d_i$ for all $i=2,\dots,n$ then $m$ divides $u_n$.
%Suppose $b_i \geq d_i$ for some $i$ and take the smallest integer $k$ such that $b_k \geq d_k$.
%Then, by the choice of $k$, we have
%\begin{align}
%    \label{eq2.2}
%    b_i \leq d_i-1 \ \ \mbox{ for }i=2,\dots,k-1.
%\end{align}
%On the other hand, since $x_k^{d_k} x_j ^{c_j} \in \init(J)$ for $j=k,k+1,\dots,n$,
%we have
%\begin{align}
%    \label{eq2.3}
%b_k \leq d_k+c_{k}-1 \ \ \mbox{ and } \ \     b_j \leq c_j-1 \ \ \mbox{ for }j=k+1,\dots,n.
%\end{align}
%Then by \eqref{eq2.2} and \eqref{eq2.3} the monomial $m$ divides $u_k$.
%\end{proof}

\begin{lemma}
    \label{lem2.3}
    With the same notation as in Lemma \ref{lem2.0}, the set
    \[\mathcal B=\{u_2,\dots,u_n\}\]
    is a $\F$-basis of $\socle\big(S/(J+(x_1))\big)$.
\end{lemma}

\begin{proof}
It follows from Lemma \ref{lem2.1} that
$\mathcal B$ is the set of maximal monomials in 
$\mathcal M(\init(J+(x_1)))$, and hence forms a $\F$-basis of
$\socle(S/\init(J+(x_1)))$.
By Lemmas \ref{lem2} and \ref{lem2.1},
the set $\mathcal B$ is linearly independent in $S/(J+(x_1))$.
Moreover, Lemma \ref{lem3} gives 
\[\dim_\F \mathrm{socle}\big(S/(J+(x_1))\big)\leq \dim_\F \mathrm{socle}\big(S/\init(J+(x_1))\big)=|\mathcal B|.\]
Thus, it remains only to prove that every $u_k$ belongs to $\mathrm{socle}\big(S/(J+(x_1))\big)$.

Fix $2 \leq k,\ell \leq n$.
We claim $x_\ell u_k =0$ in $S/(J+(x_1))$.
For this purpose, we show that
\begin{align}
\label{eq5.1}
x_s\left( \prod_{2 \leq i \leq s} x_i^{d_i-1} \right)
 x_t \left ( \prod_{ t \leq j \leq n} x_j^{c_j-1}\right) \equiv 0
 \ \ \mbox{ mod } J+(x_1)
\end{align}
for all $s,t$ with $2 \leq s \leq t \leq n$.
The desired equality $x_\ell u_k=0$ follows by taking
$(s,t)=(\ell,k)$ if $\ell\leq k$, and
$(s,t)=(k,\ell)$ if $\ell\geq k$.
%Note that $x_\ell u_k=0$ in $S/(J+(x_1))$
%follows by considering either the case $(s,t)=(\ell,k)$ or $(s,t)=(k,\ell)$.

We prove \eqref{eq5.1} by induction on $s$.
If $s=2$ then the assertion follows from
\[x_s\! \left( \prod_{2 \leq i \leq s}\! x_i^{d_i-1} \!\right)
\! x_t \!\left ( \prod_{ t \leq j \leq n} \! x_j^{c_j-1}\! \right) \! \equiv \!
 x_1^{c_1} x_{t+1}^{d_{t+1}} \! \left(\! \prod_{t <j \leq n} \! x_j^{c_j-1}\! \right) \!\equiv\! 0 \pmod {J+(x_1)}.\]
% mod $(x_2^{d_2}x_t^{c_t}-x_1^{c_1}x_{t+1}^{d_{t+1}},x_1) \subset J$.
The case $t=n$ is similar.
Suppose $s>2$ and $t<n$.
Then
\[x_s\left( \prod_{2 \leq i \leq s} x_i^{d_i-1} \right)
 x_t \left ( \prod_{ t \leq j \leq n} x_j^{c_j-1}\right) \equiv
x_{s-1}^{c_{s-1}} \left( \prod_{2 \leq i < s} x_i^{d_i-1} \right)
 x_{t+1}^{d_{t+1}} \left ( \prod_{ t < j \leq n} x_j^{c_j-1}\right) 
 \]
mod $J$ since $x_s^{d_s} x_t^{c_t}-x_{s-1}^{c_{s-1}} x_{t+1}^{d_{t+1}} \in J$.
The monomial on the right-hand side is equal to
\[
x_{s-1}^{c_{s-1}-1}x_{t+1}^{d_{t+1}-1}
\left[
x_{s-1}
\left(\prod_{2\leq i\leq s-1}x_i^{d_i-1}\right)
x_{t+1}
\left(\prod_{t+1\leq j\leq n}x_j^{c_j-1}\right)
\right].
\]
The expression in brackets belongs to $J+(x_1)$ by the induction
hypothesis, and hence so does the right-hand side.
%and this must be contained in $J+(x_1)$ by the induction hypothesis,
%as desired.
\end{proof}

We now prove our main theorem.

%\begin{theorem}
%\label{mainthm}
%Let $H=\langle a_1,\dots,a_n\rangle$ be the numerical semigroup ring minimally generated by $a_1,\dots,a_n$.
%Assume $\PF(H)=\{h+\alpha,\dots,h+(n-1)\alpha\}$
%and $h=\sum_{i=1}^n (c_i-1) a_i$ with $c_1,\dots,c_n \in \mathbb Z_{>0}$.
%If $(c_ia_i+\alpha)/a_{i+1}$ is a positive integer for all $i=1,2,\dots,n$,
%then
%\[I_H=I_2 \begin{pmatrix} x_1^{c_1} & \cdots & x_{n-1}^{c_{n-1}} & x_n^{c_n}\\ x_2^{d_2} & \cdots & x_{n}^{d_{n}} & x_1^{d_1} \end{pmatrix}\]
%with $d_{i+1}=(c_ia_i+\alpha)/a_{i+1}$ for $i=1,2,\dots,n$.
%\end{theorem}

\begin{proof}[Proof of Theorem \ref{maincor}]
As explained in the introduction, the implication
$(1)\Rightarrow(2)$ follows from \cite[Theorem 3]{KM}.
We prove $(2)\Rightarrow(1)$.
Let
\[
J=
I_2
\begin{pmatrix}
x_1^{c_1} & x_2^{c_2} & \cdots & x_n^{c_n}\\
x_2^{d_2} & x_3^{d_3} & \cdots & x_1^{d_1}
\end{pmatrix}.
\]
We will prove that $J=I_H$.

By \eqref{PFformula},
$
d_{i+1}a_{i+1}-c_i a_i=\alpha
$
for all $i$.
Hence $J$ is a graded ideal of $S$ and $J\subseteq I_H$.
By Lemma \ref{lem6}, it is therefore enough to show that
$
J+(x_1)=I_H+(x_1).
$
Consider the natural surjection
\[\psi : S/(J+(x_1)) \longrightarrow S/(I_H+(x_1)).\]
By assumption, the identities \eqref{PFformula} hold.
A direct computation from the definition of $u_k$ gives
\[
\deg u_k=a_1+h+(k-1)\alpha
\qquad (k=2,\dots,n).
\]
Therefore, under the natural isomorphism
$S/I_H\cong\F[H]$,
the set
$\mathcal B=\{u_2,\dots,u_n\}$
maps onto
$\{t^a\mid a\in a_1+\PF(H)\}$.
By Lemma \ref{lem7},  $\mathcal B$ is a $\F$-basis of
$
\socle(S/(I_H+(x_1))).
$
%
%Since \eqref{PFformula} says that 
%$\mathrm{PF}(H)$ can be written as 
%$\mathrm{PF}(H)=\{h+\alpha,h+2\alpha,\dots,h+(n-1)\alpha\}$ with
%\begin{align}
%    \label{eq:pf.1}
%    a_1+h+ (k-1) \alpha
%    &= -a_1+\left(\sum_{2 \leq i \leq k} (d_i-1)a_i \right) + a_k  + \left( \sum_{k \leq j \leq n} (c_j-1)a_j\right)
%\end{align}
%for $k=1,2,\dots,n-1$
%by the isomorphism $S/I_H \to \F[H]$
%the set $\mathcal B=\{u_2,\dots,u_n\}$ given in section 3 goes to $\{t^a \mid a \in a_1+\PF(H)\}$.
%Thus by Lemma \ref{lem7} the set $\mathcal B$ is a basis of $\socle(S/(I_H+(x_1)))$.
By Lemma \ref{lem2.3}, $\mathcal B$ is also a $\F$-basis of
$
\socle\big(S/(J+(x_1))\big).
$
Hence $\psi$ induces an isomorphism
from
$\socle\big(S/(J+(x_1))\big)$
to
$\socle\big(S/(I_H+(x_1))\big)$.

We claim that $\psi$ is injective.
Suppose, to the contrary, that
$0\neq f\in\ker\psi$.
By Lemma \ref{lem2.1}, the quotient $S/(J+(x_1))$ is
finite-dimensional over $\F$.
Then, by Lemma \ref{lem5}, there exists $m\in S$ such that
$0\neq mf\in\socle\big(S/(J+(x_1))\big)$.
However,
$\psi(mf)=m\psi(f)=0$,
contradicting the injectivity of $\psi$ on the socle.
Thus $\psi$ is injective, and hence
$J+(x_1)=I_H+(x_1)$.
Lemma \ref{lem6} now gives $J=I_H$.
\end{proof}

%Then by Lemma \ref{lem2.3} the map $\psi$ induces an isomorphism from $\socle\big(S/(J+(x_1))\big)$ to $\socle\big(S/(I_H+(x_1))\big)$.
%Such a map $\psi$ must be injective because if there is a non-zero element $f \in \mathrm{Ker}(\psi)$ then by Lemma \ref{lem5} we can choose $m \in S$ such that $0 \ne mf \in \socle\big(S/(J+(x_1))\big)$ but this contradicts that $\psi$ induces an isomorphism on socles because $\psi(mf)=0$.
%\end{proof}

\section{Proof of the second main result}

In this section, we prove our second main theorem.
Throughout this section, let
\[
H=\langle a_1,a_2,a_3,a_4\rangle
\]
be the numerical semigroup minimally generated by
$a_1,a_2,a_3,a_4$, and assume that
\[
\PF(H)=\{p_1=h+\alpha,p_2=h+2\alpha,p_3=h+3\alpha\}
\]
for some $h,\alpha\in\mathbb Z$.
We note that $h\in H$; see \cite[Lemma 2]{KM}.
We also note that the roles of $p_1$ and $p_3$ can be exchanged by
replacing $\alpha$ with $-\alpha$.
Indeed, setting $h'=h+4\alpha$, we have
$p_3=h'+(-\alpha)$,
$p_2=h'+2(-\alpha)$,
and $p_1=h'+3(-\alpha)$.
For every $w \in H$, we define
\[
h_i(w)=\max\{k\in\mathbb Z_{\geq0}\mid w-ka_i\in H\}.\]
If $w\in H$ has a unique factorization in $H$, then
\[
w=h_1(w)a_1+h_2(w)a_2+h_3(w)a_3+h_4(w)a_4.
\]
For simplicity, throughout this section we write
\[
\Ap=\Ap(H,a_1).
\]
Also, for $w\in\ZZ$, let $\overline w$ denote the unique element of
$\Ap$ satisfying
\[
\overline w\equiv w\pmod{a_1};
\]
see Lemma \ref{aperibasic}(2).
Let
\[p^+_1=a_1+p_1,\ p^+_2=a_1+p_2, \text{ and } p^+_3=a_1+p_3.\]
By Lemma \ref{aperibasic}(1),
$p_1^+,p_2^+,p_3^+$ are the maximal elements of $\Ap$
with respect to $\leq_H$.
In particular, for every $w\in\Ap$, there exists
$i\in\{1,2,3\}$ such that $w\leq_Hp_i^+$.

\begin{lemma}
    \label{4.1}
Let $w\in\Ap$ and $i\in\{2,3,4\}$ such that $h_i(w)=0$, and set
$u=\overline{w-a_i}.$
Then the following statements hold.
\begin{enumerate}
        \item[(1)] If $w-\alpha \in \Ap$ and $h_i(w-\alpha)>0$, then 
\[
u\not\leq_Hp_2^+,\
u\not\leq_Hp_3^+,\
u\not\leq_Hp_1^+-a_i,
\text{ and }
u\leq_Hp_1^+.
\]        
\item[(2)]
If $w+\alpha\in\Ap$ and $h_i(w+\alpha)>0$, then
\[
u\not\leq_Hp_1^+,\
u\not\leq_Hp_2^+,\
u\not\leq_Hp_3^+-a_i,
\text{ and }
u\leq_Hp_3^+.
\]
\end{enumerate}
\end{lemma}

\begin{proof}
Statement (2) follows from statement (1) by the symmetry obtained by
replacing $\alpha$ with $-\alpha$ and exchanging the roles of
$p_1$ and $p_3$.
Thus it suffices to prove (1).
Since $w-a_i\notin H$ and $u$ is the smallest element of $H$
congruent to $w-a_i$ modulo $a_1$, we have
$u>w-a_i$.
Since $u\equiv w-a_i \pmod {a_1}$, we have
    \begin{itemize}
       \item [(i)] $(w-\alpha-a_i) +(p_2^+-u) \equiv p_1^+ \pmod {a_1}$ and $(w-\alpha-a_i) +(p_2^+-u) < p_1^+$.
        Since $p_1^+\in\Ap$, Lemma \ref{aperibasic}(2) implies that no
smaller integer congruent to $p_1^+$ modulo $a_1$ belongs to $H$.
Hence  $(w-\alpha -a_i)+(p_2^+-u) \not \in H$.
        Since $w-\alpha-a_i \in H$, it follows that $p_2^+-u \not \in H.$
       \item [(ii)] $(w-\alpha-a_i) +(p_3^+-u) \equiv p_2^+ \pmod {a_1}$ and $(w-\alpha-a_i) +(p_3^+-u) < p_2^+$.  Arguing as in (i), we have $(w-\alpha -a_i)+(p_3^+-u) \not \in H$ and hence $(p_3^+-u) \not \in H$.
       \item[(iii)] $w +(p_1^+-a_i-u) \equiv p_1^+ \pmod {a_1}$ and $w +(p_1^+-a_i-u) <p_1^+$.
Hence $w+(p_1^+-a_i-u) \not \in H$.
Since $w \in \Ap \subset H$, it follows that  $(p_1^+-a_i-u) \not \in H$.
    \end{itemize}
Statements (i)--(iii) show that $u \not \leq_H p_2^+$, $u \not \leq_H p_3^+$ and $u \not \leq_H p_1^+-a_i$.
Since $u\leq_H p_j^+$ for some $j$,
we must have $u\leq_H p_1^+$.
\end{proof}

\begin{corollary}
\label{cor4}
Let $w\in\Ap$ and $i\in\{2,3,4\}$.
If $w-\alpha,w+\alpha\in \Ap$, then
\[
h_i(w) \geq \min \{h_i(w-\alpha),h_i(w+\alpha)\}.
\]
\end{corollary}

\begin{proof}
Let $h=h_i(w)$.
Suppose, to the contrary, that
\[h< \min\{h_i(w-\alpha),h_i(w+\alpha)\}.\]
By the definition of $h$ and Lemma \ref{aperibasic}(3),
$w-\alpha-h a_i,w-ha_i,w+\alpha-ha_i$ are all elements of $\Ap$.
Also, 
$h_i(w-\alpha-h
a_i)>0$, $h_i(w+\alpha-h a_i)>0$ and $h_i(w-ha_i)=0$.
Set $u=\overline {(w-ha_i-a_i)}$.
Applying Lemma \ref{4.1} to $w-ha_i$, we obtain $u \not \leq_H p_1^+,$ $u \not \leq_H p_2^+$ and $u \not \leq_H p_3^+$.
This contradicts the maximality statement in
Lemma \ref{aperibasic}(1).
\end{proof}

\subsection*{Proof of Theorem \ref{secondthm}}
We now prove Theorem \ref{secondthm}.
After relabeling the generators if necessary, we may assume that
$\Ap=\Ap(H,a_1)$ has unique factorizations in $H$.
We write
\begin{align*}
p_1^+ &= z_{12}a_2+z_{13}a_3+z_{14}a_4\\
p_2^+ &= z_{22}a_2+z_{23}a_3+z_{24}a_4\\
p_3^+ &= z_{32}a_2+z_{33}a_3+z_{34}a_4
\end{align*}
with $z_{ij} \in \ZZ_{\geq 0}$.
We note that $z_{ij}=h_j(p_i^+)$.
We need the following technical lemma.

\begin{lemma}
\label{4.2}\ 
\begin{enumerate}
\item
If $z_{12}>z_{22}\geq z_{32}$, $z_{14}\leq z_{24}<z_{34}$, and $z_{13}<z_{23}$,
then
$\alpha\equiv (z_{22}+1)a_2\pmod {a_1}$ and $z_{22}=z_{32}$.
\item
If $z_{12}>z_{22}\geq z_{32}$, $z_{14}\leq z_{24}<z_{34}$, and $z_{33}<z_{23}$,
then
$-\alpha\equiv (z_{24}+1)a_4\pmod {a_1}$ and $z_{24}=z_{14}$.
\end{enumerate}
\end{lemma}

\begin{proof}
The statement (2) is equivalent to the statement (1) if we replace $\alpha$ with $-\alpha$ and exchange $a_2$ and $a_4$.
Hence it suffices to prove (1).

Consider the set
\[X=\{(s,t) \in \mathbb Z^2_{\geq 0} \mid s \leq z_{13},t \leq z_{14}\}.
\]
For each $(s,t) \in X$, the elements
\[p^+_1-(z_{22}a_2+sa_3+ta_4) \text{ and }
p^+_2-(z_{22}a_2+sa_3+ta_4)
\]
belong to $\Ap$ by Lemma \ref{aperibasic}(3).
Since $p^+_1=p^+_2-\alpha$ and $h_2\big(p^+_2-(z_{22}a_2+sa_3+ta_4)\big)=0$,
by Lemma \ref{4.1}(1) if we set
\begin{align}
\label{ust}
u_{s,t}=\overline{p_2^+-((z_{22}+1)a_2+sa_3+ta_4)},
\end{align}
then
$u_{s,t} \leq _H p_1^+$ and $u_{s,t} \not \leq _H p_1^+-a_2$.
Since $u_{s,t}\leq_Hp_1^+$ and both $u_{s,t}$ and $p_1^+$
have unique factorizations, the coefficients in the factorization of
$u_{s,t}$ are bounded above by those of $p_1^+$.
Moreover, $u_{s,t}\not\leq_Hp_1^+-a_2$ forces the coefficient
of $a_2$ to be $z_{12}$.
Hence
\[
u_{s,t}=z_{12}a_2+\theta_{s,t} a_3 + \eta_{s,t} a_4
\]
for some  $(\theta_{s,t},\eta_{s,t}) \in X$.
We claim 
$u_{z_{13},z_{14}}=z_{12}a_2$.

We first prove
\[
\theta_{z_{13},t}=0
\qquad\text{for all } t=0,1,\dots,z_{14}.
\]
Let $s<z_{13}$. By the definition of $u_{s,t}$,
\[
u_{s+1,t}+a_3\equiv u_{s,t}\pmod{a_1},
\]
and hence
\[
(\theta_{s+1,t}+1)a_3+\eta_{s+1,t}a_4
\equiv
\theta_{s,t}a_3+\eta_{s,t}a_4
\pmod{a_1}.
\]
The right-hand side belongs to $\Ap$, since it is less than or equal
to $p_1^+$ with respect to $\leq_H$. The left-hand side also belongs
to $\Ap$, since
$\theta_{s+1,t}+1\leq z_{13}+1\leq z_{23}$
and
$\eta_{s+1,t}\leq z_{14}\leq z_{24}$.
Thus the two sides are equal. Since $\Ap$ has unique factorizations,
we obtain
$\theta_{s,t}=\theta_{s+1,t}+1$ and 
$\eta_{s,t}=\eta_{s+1,t}$.
Consequently, we have
\[
\theta_{z_{13},t}=\theta_{z_{13}-1,t}-1
=\cdots =\theta_{0,t}-z_{13}.
\]
Since $0 \leq \theta_{0,t} \leq z_{13}$,
we get
$\theta_{z_{13},t}=0$ as desired.

We next determine $\eta_{z_{13},z_{14}}$.
Since $\theta_{z_{13},t}=0$ for every $t$, we have
$u_{z_{13},t}
=
z_{12}a_2+\eta_{z_{13},t}a_4.$
Let $t<z_{14}$. By the definition of $u_{s,t}$,
we have $u_{z_{13},t+1}+a_4
\equiv
u_{z_{13},t}
\pmod{a_1}$.
Hence
\[
(\eta_{z_{13},t+1}+1)a_4
\equiv
\eta_{z_{13},t}a_4
\pmod{a_1}.
\]
Since
$\eta_{z_{13},t+1}+1
\leq z_{14}+1
\leq z_{34}$,
both
$(\eta_{z_{13},t+1}+1)a_4$
and $\eta_{z_{13},t}a_4$
are less than or equal to $p^+_3$ with respect to $\leq_H$.
Therefore they must be equal and we obtain
\[
\eta_{z_{13},t}
=
\eta_{z_{13},t+1}+1 \text{ for }t <z_{14}.
\]
But since $0 \leq \eta_{z_{13},t}\leq z_{14}
$ the above equality guarantees
$\eta_{z_{13},z_{14}}=0$,
and hence
$u_{z_{13},z_{14}}=z_{12}a_2$
as desired.

We now complete the proof.
Recall $u_{s,t}=\overline{(p_2^+-(z_{22}+1)a_2-sa_3-ta_4)}$. Then
\[ p^+_1=u_{z_{13},z_{14}}+z_{13}a_3+z_{14}a_4 \equiv p_2^+-(z_{22}+1)a_2 \pmod {a_1}.
\]
Since $p_2^+=p_1^++\alpha$ this proves the first statement
$\alpha \equiv (z_{22}+1)a_2 \pmod {a_1}$.
Next, we prove $z_{22}=z_{32}$.
If $z_{22}>z_{32}$, then
\[p^+_2-z_{32}a_2=p^+_3-z_{32}a_2-\alpha \text{ and } p^+_3-z_{32}a_2\]
are elements of $\Ap$ and $h_2(p^+_3-z_{32}a_2)=0$.
Hence by Lemma \ref{4.1}(1) we have
$\overline{p^+_3-(z_{32}+1)a_2}\leq_H p^+_1$ and
$\overline{p^+_3-(z_{32}+1)a_2} \not \leq _H p_1^+-a_2$.
Such an element must be of the form
\[z_{12}a_2+s a_3+ta_4\]
for some $(s,t) \in X$.
%This implies
%\[
%p^+_2 \equiv p^+_3-\alpha \equiv p^+_3-(z_{22}+1)a_2
%\equiv 
%p^+_3-(z_{32}+1)a_2 -(z_{22}-z_{32})a_2
%\equiv (z_{12}-(z_{22}-z_{32}))a_2+sa_3+ta_4
%\]
This implies
\begin{align*}
p_2^+
&\equiv p_3^+-\alpha\\
&\equiv p_3^+-(z_{22}+1)a_2\\
&\equiv
\bigl(z_{12}-(z_{22}-z_{32})\bigr)a_2
+sa_3+ta_4
\pmod{a_1}.
\end{align*}
The last expression must belong to $\Ap$ 
since it is $\leq_H p_1^+$.
But this means that this element equals $p^+_2$ since $\Ap$ contains exactly one element in its congruence class modulo $a_1$,
which contradicts the maximality of $p_2^+$.
Hence we have $z_{22}=z_{32}$.
\end{proof}

We now use Theorem \ref{maincor} to prove that $I_H$ is generated
by $2\times2$ minors.
Since $p_1^+,p_2^+,p_3^+$ are pairwise incomparable
with respect to $\leq_H$,
we have $z_{1j}>z_{2j}$ for some $j$.
After relabeling $a_2,a_3,a_4$ if necessary, we may assume that
$z_{12}>z_{22}$.
Then by Corollary \ref{cor4} we have $z_{22}\geq z_{32}$.
Since $p^+_2$ and $p^+_3$ are incomparable with respect to $\leq _H$,
we have $z_{33}>z_{23}$ or $z_{34}>z_{24}$.
By exchanging $a_3$ and $a_4$ if necessary,
we may assume $z_{34}>z_{24}$.
In particular, in this setting we have $z_{14} \leq z_{24}$ by Corollary \ref{cor4}.
To summarize, we can assume
\begin{align}
\label{katei}
z_{12}> z_{22} \geq z_{32} \text{ and } z_{14} \leq z_{24} < z_{34}.
\end{align}

[Step 1]
We prove that the case
$z_{13}=z_{23}=z_{33}$ cannot occur.
Suppose $z_{13}=z_{23}=z_{33}$.
Since $p_1^+,p_2^+$, and $p_3^+$ are incomparable with respect to $\leq _H$,
we must have
\[
z_{12}> z_{22} > z_{32},\
\ z_{13}=z_{23}=z_{33},
\text{ and }
z_{14} < z_{24} < z_{34}.
\]
Let $\alpha_3=\min \{k \in \mathbb Z_{>0}\mid k a_3 \in \langle a_1,a_2,a_4\rangle\}$ and assume that $\alpha_3 a_3= \tau_1 a_1+ \tau_2 a_2 + \tau_4 a_4$ for some $\tau_1,\tau_2,\tau_4 \in \mathbb Z_{\geq 0}$.
Since $p_1^+,p_2^+,p_3^+$ have unique factorizations in $H$,
we must have $z_{13}=z_{23}=z_{33}<\alpha_3$.
We actually have
\[ z_{13}=z_{23}=z_{33}=\alpha_3-1.\]
Indeed, since $(\alpha_3-1)a_3 \not \in \langle a_1,a_2,a_4\rangle$, we must have $(\alpha_3-1) a_3 \in \Ap=\Ap(H,a_1)$.
Hence $(\alpha_3-1)a_3 \leq_H p^+_i$ for some $i$,
which proves $z_{13}=z_{23}=z_{33}\geq \alpha_3-1$.

Since $p_1,p_3 \in \PF(H)$,
$p_1+a_4\in \Ap(H,a_4)$ and $p_3+a_2 \in \Ap(H,a_2)$. Thus
\begin{align*}
p_1^+-a_1+a_4&=p_1+a_4=k_1a_1+k_2a_2+k_3a_3,\\
p_3^+-a_1+a_2&=p_3+a_2=k'_1a_1+k'_3a_3+k'_4a_4
\end{align*}
for some $k_1,k_2,k_3,k_1',k_3',k_4' \in \mathbb Z_{\geq 0}$.
Using the relation $\alpha_3 a_3= \tau_1a_1+\tau_2 a_2+\tau_4a_4$,
we can take such integers with $k_3<\alpha_3$ and $k_3' <\alpha_3$.
Substituting $p_1^+=z_{12}a_2+z_{13}a_3+z_{14}a_4$ and $p_3^+=z_{32}a_2+z_{33}a_3+z_{34}a_4$ into the above equations, we have
\begin{align}
    \label{hennkou2}
    (z_{12}-k_2)a_2+(z_{13}-k_3)a_3+(z_{14}+1)a_4 =(k_1+1)a_1\\
    \label{hennkou3}
    (z_{32}+1)a_2+(z_{33}-k'_3)a_3+(z_{34}-k_4')a_4 =(k'_1+1)a_1.
    \end{align}
If $(z_{12}-k_2,z_{13}-k_3,z_{14}+1) \leq (z_{22},z_{23},z_{24})$ then we can rewrite $p_2^+=z_{22}a_2+z_{23}a_3+z_{24}a_4$ in the form
\[p_2^+=(k_1+1)a_1+\ell_2a_2+\ell_3a_3+\ell_4a_4\]
with $\ell_2,\ell_3,\ell_4 \in \mathbb Z_{\geq 0}$, which contradicts $p_2^+ \in \Ap(H,a_1)$.
Since $z_{13}-k_3\leq z_{23}$ and
$z_{14}+1\leq z_{24}$, it follows that $z_{12}-k_2>z_{22}$.
Similarly, by comparing $(z_{32}+1,z_{33}-k_3',z_{34}-k_4')$ and $(z_{22},z_{23}=z_{33},z_{24})$ we have $z_{34}-k_4' >z_{24}$.
Hence
\begin{align}
    \label{hennkou4}
    z_{12}-k_2 > z_{22} \geq z_{32}+1
    \text{ and } z_{34}-k_4' >z_{24} \geq z_{14}+1.
\end{align}
By taking the difference between \eqref{hennkou2} and \eqref{hennkou3} we have
\begin{align}
    \label{hennkou5}
    \big((z_{12}\!-\!k_2)\!-\!(z_{32}\!+\!1)\big)a_2
    \!+\!(k_3'\!-\!k_3)a_3\!=\!(k_1\!-\!k_1')a_1\!+\!\big( (z_{34}\!-\!k_4')\!-\!(z_{14}\!+\!1)\big)a_4.\end{align}
We conclude a contradiction using \eqref{hennkou4} and \eqref{hennkou5}.
Observe that $|k_3'-k_3|<\alpha_3$.
Since $((z_{12}-k_2)-(z_{32}+1),k_3'-k_3) \leq (z_{12},z_{13}=\alpha_3-1)$,
if $k_1-k_1'\geq0$, then \eqref{hennkou4} and \eqref{hennkou5}
give two distinct factorizations of $p_1^+$, a contradiction (note that $k_3'-k_3$ could be negative).
On the other hand,
since $(k_3-k'_3,(z_{34}-k'_4)-(z_{14}+1)) \leq (z_{33}=\alpha_3-1,z_{34})$,
if $k_1'-k_1\geq0$, then \eqref{hennkou4} and \eqref{hennkou5}
give two distinct factorizations of $p_3^+$, again a contradiction.
This proves that $z_{13}=z_{23}=z_{33}$ cannot occur.
\bigskip

[Step 2] 
We next reduce the proof to the case where
\begin{align}
    \label{desired}
z_{12}>z_{22}= z_{32},\
z_{13}<z_{23}>z_{33}, \text{ and }
z_{14}= z_{24}< z_{34}.
\end{align}
We proved in Step 1 that $z_{13},z_{23},z_{33}$ are not all equal.
Since replacing $\alpha$ with $-\alpha$ and simultaneously
exchanging $a_2$ and $a_4$ preserves \eqref{katei} while exchanging
the roles of $z_{13}$ and $z_{33}$, we may assume
$z_{13}\ne z_{23}$.
We prove that we can assume $z_{13}<z_{23}$.

Suppose that $z_{13}>z_{23}$. Then by Corollary \ref{cor4}
we have $z_{13}>z_{23} \geq z_{33}$
and therefore
\[
z_{12}>z_{22}\geq z_{32},\
z_{13}>z_{23}\geq z_{33}, \text{ and }
z_{14}\leq z_{24}< z_{34}.
\]
However, since $p^+_2$ and $p^+_3$ are incomparable with respect to $\leq_H$,
we have either $z_{22}\ne z_{32}$ or $z_{23} \ne z_{33}$.
By exchanging $a_2$ and $a_3$ if necessary,
we may assume $z_{23}\ne z_{33}$.
In this setting, we have
\[
z_{12}>z_{22}\geq z_{32},\
z_{13}>z_{23}> z_{33}, \text{ and }
z_{14}\leq z_{24}< z_{34}.
\]
But if we exchange $\alpha$ and $-\alpha$ as well as $a_2$ and $a_4$,
the situation becomes
\[
z_{12}>z_{22}\geq z_{32},\
z_{13}<z_{23}< z_{33}, \text{ and }
z_{14}\leq z_{24}< z_{34}.
\]
(We actually soon see that this situation cannot occur.)
Hence we may assume both \eqref{katei} and
$z_{13}<z_{23}$.

We now show the desired condition \eqref{desired}.
By \eqref{katei} and
$z_{13}<z_{23}$, we have
\[
z_{12}>z_{22}\geq z_{32},\
z_{13}<z_{23}, \text{ and }
z_{14}\leq z_{24}< z_{34}.
\]
Then by Lemma \ref{4.2}(1)
we have $\alpha \equiv (z_{22}+1)a_2 \pmod{a_1}$
and
\[
z_{12}>z_{22}= z_{32},\
z_{13}<z_{23}, \text{ and }
z_{14}\leq z_{24}< z_{34}.
\]
Since $p_2^+$ and $p_3^+$ are incomparable with respect to $\leq_H$,
the above conditions say $z_{23}>z_{33}$.
Then by Lemma \ref{4.2}(2),
we have $-\alpha \equiv (z_{24}+1)a_4 \pmod {a_1}$
and we get the desired condition \eqref{desired}.

We finally prove that condition (2) in Theorem \ref{maincor} is satisfied (and therefore $I_H$ is generated by $2$-minors).
Recall that $p_1=h+\alpha$, $p_2=h+2\alpha $, and $p_3=h+3\alpha$.
By \eqref{desired},
there exist positive integers
$c_2,c_3,c_4,d_2,d_3,d_4,\beta$ such that
\begin{align}
\label{-1-}
\begin{array}{ll}
h+\alpha&= p_1 =-a_1+ (c_2+d_2-1)a_2+(c_3-1)a_3+(c_4-1)a_4,\\
h+2\alpha&= p_2 = -a_1+ (d_2-1)a_2+(\beta-1) a_3+(c_4-1)a_4,\\
h+3\alpha&=  p_3 =-a_1+ (d_2-1)a_2+(d_3-1)a_3+(c_4+d_4-1)a_4.
\end{array}
\end{align}
Since $\alpha\equiv d_2a_2 \pmod {a_1}$ by Lemma \ref{4.2}(1),
we also have
\[
h+a_1=p_1^+-\alpha \equiv (c_2-1)a_2+(c_3-1)a_3+(c_4-1)a_4 \pmod {a_1}.
\]
Since $h \in H$ we have $p_1^+-\alpha=h+a_1 \in H$.
Then since the right-hand side of the above equation belongs to $\Ap$
and since $\alpha<d_2a_2$ (as $\alpha\equiv d_2a_2 \pmod {a_1}$, $\alpha \not \in H$ and $d_2a_2 \in \Ap$) we must have
\begin{align}
\label{-2-}
h+a_1=p_1^+-\alpha = c_1 a_1 +(c_2-1)a_2+(c_3-1)a_3+(c_4-1)a_4
\end{align}
for some positive integer $c_1$.
Similarly,
since $-\alpha \equiv c_4a_4 \pmod{a_1}$, we have
\begin{align}
    \label{-3-}
h+4\alpha+a_1=p^+_3+\alpha =(d_2-1)a_2+(d_3-1)a_3+(d_4-1)a_4+d_1a_1\end{align}
for some positive integer $d_1$.
Then, by Theorem \ref{maincor} and equations \eqref{-1-}, \eqref{-2-} and \eqref{-3-},
to complete the proof of the theorem
what we must prove is $\beta=c_3+d_3$.

We first note that by equations \eqref{-1-}, \eqref{-2-} and \eqref{-3-},
we have
\begin{align}
\label{iti}
\alpha&=p_1-(p_1-\alpha)=d_2a_2-c_1a_1,\\
\label{ni}
\alpha&=p_2-p_1=(\beta-c_3)a_3-c_2a_2,\\
\label{sann}
\alpha&=p_3-p_2=d_4a_4-(\beta-d_3)a_3,\\
\label{shi}
\alpha&=(p_3+\alpha)-p_3=d_1a_1-c_4a_4.
\end{align}
By \eqref{ni} and \eqref{sann},
we have
\[
(2\beta-c_3-d_3)a_3=c_2a_2+d_4a_4.
\]
If $(2\beta-c_3-d_3)<\beta,$ then the factorization of $p^+_2$ in $H$ is not unique, so we must have $\beta \geq c_3+d_3$.
Set
\[q=p_1^+-(c_3-1)a_3-(c_4-1)a_4=(c_2+d_2-1)a_2 \text{ and } u=\overline{q-a_3}.\]
Since $q,q+\alpha \in \Ap$, $h_3(q)=0$, and $h_3(q+\alpha)>0$,
by Lemma \ref{4.1}(2) we have $u \leq_H p_3^+$, $u \not \leq _H p^+_1, u \not \leq_H p^+_2$ and $u \not \leq_H p^+_3-a_3$.
Hence $u$ can be written in the form
\[
u=sa_2+(d_3-1)a_3+ta_4 \ \ \text{with } s\leq d_2-1 \text{ and } c_4 \leq t \leq c_4+d_4-1
\]
On the other hand, since
\[
u \equiv q-a_3=(c_2+d_2-1)a_2-a_3\pmod {a_1}
\]
and $u>q-a_3$,
there is a positive integer $\sigma$ such that
\[
sa_2+(d_3-1)a_3+ta_4=u=(c_2+d_2-1)a_2-a_3+\sigma a_1
\]
and we get
\[(c_2+d_2-1-s)a_2+\sigma a_1=d_3a_3+ta_4.
\]
Consider the following equation which follows from \eqref{ni} and \eqref{shi}
\[
c_2a_2+d_1a_1=(\beta-c_3)a_3+c_4a_4.
\]
Taking the difference between the above two equations, we obtain
\begin{align}
\label{lastdiff}
(d_2-1-s)a_2+(\beta-c_3-d_3)a_3=(t-c_4)a_4+(d_1-\sigma)a_1.
\end{align}
Observe that the coefficients $d_2-1-s,\beta-c_3-d_3,t-c_4$ are non-negative. Moreover
\[
(d_2-1-s)a_2+(\beta-c_3-d_3)a_3\leq_H p^+_2
\text{ and }
 (t-c_4)a_4 \leq_H p^+_3,
 \]
so both elements on the left belong to $\Ap$.
It follows that $d_1-\sigma=0$.
Indeed, if $d_1-\sigma>0$, then \eqref{lastdiff} implies that the
left-hand side does not belong to $\Ap$; if $d_1-\sigma<0$, then
the right-hand side does not belong to $\Ap$.
Hence
\[
(d_2-1-s)a_2+(\beta-c_3-d_3)a_3=(t-c_4)a_4
\]
but 
since every element of $\Ap$ has a unique factorizations in $H$,
the last equality forces
$d_2-1-s=0,\beta-c_3-d_3=0$ and $t-c_4=0$.
This proves $\beta=c_3+d_3$,
completing the proof of Theorem \ref{secondthm}.

\begin{remark}
Numerical semigroups having an Ap\'ery set of unique expression were studied by Rosales \cite{Rosales}. In particular, it was shown in \cite{Rosales} that, under this assumption, a minimal presentation of the numerical semigroup can be described in terms of the minimal elements lying outside the Ap\'ery set. Our proof of Theorem \ref{secondthm} first obtains explicit presentations of the maximal elements of $\Ap(H,a_1)$ and then uses these presentations to control the defining ideal. Rosales' result appears to be closely related to the second step of our proof, although we have not investigated whether it can be applied directly in our setting.
\end{remark}

\section*{Final remark and examples}
Although $\Ap(H,a_i)$ does not always have a unique factorization,
we still think that the approach of this paper is useful to study Conjecture \ref{conj} at least for the case when $n=4$.
For example, Lemma \ref{uniqueAP} suggests that if the conjecture holds then by an appropriate choice of $i$ one should have (i) $a_i+p_1$ and $a_i+p_2$ have the unique factorizations in $H$,
or (ii) $a_i+p_2$ and $a_i+p_3$ have the unique factorizations in $H$. If one can prove such a uniqueness statement only assuming that $\PF(H)$ forms an arithmetic progression of length $3$,
some of the argument in this paper will be applicable to attach the $n=4$ case of the conjecture.

We conclude the paper with two examples of four-generated numerical semigroups whose pseudo-Frobenius numbers form an arithmetic progression, one for which the Apéry set has unique factorizations and one for which it does not.

\begin{example}
    Let $H=\langle a_1=17,a_2=40,a_3=42,a_4=45\rangle$.
    Then
    $\PF(H)=\{145,151,157\}$
    with $h=139$ and $\alpha=6$.
    Thus, we have
    \begin{equation*}
        I_H=\mathrm{I}_2\left(\begin{matrix}
    X_1^2&X_2^3&X_3^2&X_4\\
    X_2&X_3^3&X_4^2&X_1^3
    \end{matrix}\right).
    \end{equation*}
    Since $1=d_2\leq c_2=3$ and $2=d_4 \geq c_4=1$, by Lemma \ref{uniqueAP}, this Ap\'ery set has unique factorizations in $H$.
\end{example}

\begin{example}
    Let $H=\langle a_1=7,a_2=15,a_3=16,a_4=17\rangle$.
    Then
    $\PF(H)=\{25,26,27\}$
    with $h=24$ and $\alpha=1$.
    We have
    \begin{equation*}
        I_H=I_2\left(\begin{matrix}
    X_1^2&X_2&X_3&X_4^2\\
    X_2&X_3&X_4&X_1^5
    \end{matrix}\right).
    \end{equation*}
    The element $32$ of $\Ap$ has the following two factorizations,
    \begin{equation*}
        32 = a_2+a_4 = 2a_3.
    \end{equation*}
    Thus, this Ap\'ery set does not have unique factorizations in $H$.    
\end{example}

\noindent
\textbf{Acknowledgments}:
The first author is partly supported by KAKENHI 25K06943.
We would like to thank Shumpei Higuchi for developing the computer program used for the computational experiments in this research.

\end{document}